\documentclass[11pt,twoside]{article} 
\usepackage{bbm}
\usepackage{amsmath,amssymb}
\makeatletter
\newcommand*{\mint}[1]{%
  \mint@l{#1}{}%
}
\newcommand*{\mint@l}[2]{%
  \@ifnextchar\limits{%
    \mint@l{#1}%
  }{%
    \@ifnextchar\nolimits{%
      \mint@l{#1}%
    }{%
      \@ifnextchar\displaylimits{%
        \mint@l{#1}%
      }{%
        \mint@s{#2}{#1}%
      }%
    }%
  }%
}
\newcommand*{\mint@s}[2]{%
  \@ifnextchar_{%
    \mint@sub{#1}{#2}%
  }{%
    \@ifnextchar^{%
      \mint@sup{#1}{#2}%
    }{%
      \mint@{#1}{#2}{}{}%
    }%
  }%
}
\def\mint@sub#1#2_#3{%
  \@ifnextchar^{%
    \mint@sub@sup{#1}{#2}{#3}%
  }{%
    \mint@{#1}{#2}{#3}{}%
  }%
}
\def\mint@sup#1#2^#3{%
  \@ifnextchar_{%
    \mint@sup@sub{#1}{#2}{#3}%
  }{%
    \mint@{#1}{#2}{}{#3}%
  }%
}
\def\mint@sub@sup#1#2#3^#4{%
  \mint@{#1}{#2}{#3}{#4}%
}
\def\mint@sup@sub#1#2#3_#4{%
  \mint@{#1}{#2}{#4}{#3}%
}
\newcommand*{\mint@}[4]{%
  \mathop{}%
  \mkern-\thinmuskip
  \mathchoice{%
    \mint@@{#1}{#2}{#3}{#4}%
        \displaystyle\textstyle\scriptstyle
  }{%
    \mint@@{#1}{#2}{#3}{#4}%
        \textstyle\scriptstyle\scriptstyle
  }{%
    \mint@@{#1}{#2}{#3}{#4}%
        \scriptstyle\scriptscriptstyle\scriptscriptstyle
  }{%
    \mint@@{#1}{#2}{#3}{#4}%
        \scriptscriptstyle\scriptscriptstyle\scriptscriptstyle
  }%
  \mkern-\thinmuskip
  \int#1%
  \ifx\\#3\\\else_{#3}\fi
  \ifx\\#4\\\else^{#4}\fi
}
\newcommand*{\mint@@}[7]{%
  \begingroup
    \sbox0{$#5\int\m@th$}%
    \sbox2{$#5\int_{}\m@th$}%
    \dimen2=\wd0 %
    \let\mint@limits=#1\relax
    \ifx\mint@limits\relax
      \sbox4{$#5\int_{\kern1sp}^{\kern1sp}\m@th$}%
      \ifdim\wd4>\wd2 %
        \let\mint@limits=\nolimits
      \else
        \let\mint@limits=\limits
      \fi
    \fi
    \ifx\mint@limits\displaylimits
      \ifx#5\displaystyle
        \let\mint@limits=\limits
      \fi
    \fi
    \ifx\mint@limits\limits
      \sbox0{$#7#3\m@th$}%
      \sbox2{$#7#4\m@th$}%
      \ifdim\wd0>\dimen2 %
        \dimen2=\wd0 %
      \fi
      \ifdim\wd2>\dimen2 %
        \dimen2=\wd2 %
      \fi
    \fi
    \rlap{%
      $#5%
        \vcenter{%
          \hbox to\dimen2{%
            \hss
            $#6{#2}\m@th$%
            \hss
          }%
        }%
      $%
    }%
  \endgroup
}
\usepackage{mathrsfs}
\usepackage{amssymb}
\usepackage{amsmath}
\usepackage{amsthm}
\usepackage{amsfonts}
\usepackage{color}
\usepackage{graphicx}
\usepackage[active]{srcltx}
\usepackage{tikz}
\usepackage{pgflibraryarrows}
\usepackage{pgflibrarysnakes}

\usepackage[cp1252]{inputenc}

\usepackage{mathrsfs}
\usepackage{graphicx}

\usepackage[active]{srcltx}

\allowdisplaybreaks

\usepackage{titletoc}
\titlecontents{section}[0pt]{\addvspace{2pt}\filright}
              {\contentspush{\thecontentslabel\ }}
              {}{\titlerule*[8pt]{.}\contentspage}

\def\rr{{\mathbb R}}
\def\rn{{{\rr}^n}}

\def\fz{\infty}
\def\az{\alpha}

\def\dist{{\mathop\mathrm{\,dist\,}}}
\def\loc{{\mathop\mathrm{\,loc\,}}}

\def\lz{\lambda}
\def\dz{\delta}
\def\bdz{\Delta}
\def\ez{\epsilon}

\def\bz{\beta}

\def\gz{{\gamma}}

\def\wz{\widetilde}

\def\bint{{\ifinner\rlap{\bf\kern.35em--}
\int\else\rlap{\bf\kern.45em--}\int\fi}\ignorespaces}

\def\bbint{{\ifinner\rlap{\bf\kern.35em--}
\hspace{0.078cm}\int\else\rlap{\bf\kern.45em--}\int\fi}\ignorespaces}

\def\esssup{{\rm \,esssup\,}}
\def\essinf{{\rm \,essinf\,}}

\newtheorem{thm}{Theorem}[section]
\newtheorem{lem}[thm]{Lemma}
\newtheorem{rem}[thm]{Remark}
\numberwithin{equation}{section}

\title
{\Large\bf  Interior H\"older regularity for stable solutions to
semilinear elliptic equations up to dimension 5
\footnotetext{\hspace{-0.35cm}
\endgraf
 2020 {\it Mathematics Subject Classification:} 35J61, 35B65, 35B35.
 \endgraf {\it Key words and phrases:}   semilinear equation, stable solution, H\"older regularity
\endgraf
 The first and second authors are funded by the Chinese Academy of Science and NSFC grant No. 11688101.
The third author is supported by NSFC (No. 11871088 \& No.12025102) and by the Fundamental Research Funds for the Central Universities.
\endgraf $^\ast$ Corresponding author.}
 }

\author{Fa Peng$^\ast$, Yi Ru-Ya Zhang  and Yuan Zhou}
\begin{document}

\arraycolsep=1pt
\allowdisplaybreaks
 \maketitle

\begin{center}
\begin{minipage}{13.5cm}\small
 \noindent{\bf Abstract.}\quad
Let  $2\le n\le 5$.
 We establish an apriori
 interior H\"older regularity of $C^2$-stable  solutions to  the semilinear equation
 $-\bdz u=f(u)$ in any domain of $\rn$ for  any  nonlinearity $f\in C^{0,1}(\rr) $.
 If   $f $ is nondecreasing and convex in addition,
we  obtain an interior H\"older regularity, and hence the local boundedness, of
$W^{1,2}_\loc(\Omega)$-stable solutions  by  locally approximating
them  via $C^2(\Omega)$-stable  solutions. In particular, we
do not require any lower bound on $f$.

\end{minipage}
\end{center}

\section{Introduction}
Let $n\ge 2$ and $\Omega\subset\rn $ be a bounded
 domain.
Consider the  semilinear elliptic equation
\begin{equation}\label{sta}
-\bdz u=f(u)\quad {\rm in}\ \Omega,
\end{equation}
where  the nonlinearity $f\in C^{0,1}(\rr)$, that is, $f$ is  locally Lipschitz in $\rr$.
The equation \eqref{sta} is the Euler-Lagrange equation for minimizers of the energy functional
$$\mathcal{E}(u,\Omega):=\frac12\int_{\Omega}|Du|^2\,dx-\int_{\Omega} F(u)\,dx,$$
 where $F(t)=\int^t_0f(t)\,dt$ for $t\in\rr$.
 A function $u:\Omega\to\rr$  is called as  a
$L^1_{\loc}(\Omega)$-weak solution
to the equation \eqref{sta}   if $u\in L^1_\loc(\Omega)$, $f(u)\in L^1_\loc(\Omega)$ and
\begin{equation}\label{stt00} -\int_{\Omega} u\Delta\xi\,dx=\int_{\Omega}f(u)\xi\,dx\quad \forall
\xi\in C^\fz_c(\Omega).\end{equation}
 If    $u$ is a  $W^{1,2}_\loc(\Omega)$-weak solution, that is, $u\in W^{1,2}_\loc(\Omega)$  apriori and $u$ is a $L^1_{\loc}(\Omega)$-weak solution  to  \eqref{sta} in the sense of \eqref{stt00}, then
  \eqref{stt00} reads as
\begin{equation}\label{stt0}\int_{\Omega}Du\cdot D\xi\,dx=\int_{\Omega}f(u)\xi\,dx\quad \forall
\xi\in C^\fz_c(\Omega).\end{equation}
Moreover, write $$f'_{-}(t):=\lim_{h\to 0^{+}}\frac{f(t)-f(t-h)}{h}\quad \forall t\in \rr,$$
which, when $ f\in C^1(\rn)$,  coincide with $f'(t)$.
A $L^1_{\loc}(\Omega)$-weak solution $u$ to \eqref{sta} is  called as  a stable solution   if
$f'_{-}(u)\in L^1_{\loc}(\Omega)$ and
\begin{equation}\label{stt}
\int_\Omega |D\xi|^2\,dx-\int_\Omega f'_{-}(u)\xi^2\,dx\ge 0\quad \forall \xi\in C^{0,1}_c(\Omega).
\end{equation}
  Note that  \eqref{stt} means the nonnegativity of the second variation of the energy $\mathcal{E}(u,\Omega)$  at $u$, and moreover, \eqref{stt} coincides with the nonnegativity of
 the first Dirichlet eigenvalue of the linearized operator $-\bdz -f'_{-}(u)$ in $\Omega$, for more details see \cite{d11}.

In dimension $n\le 9$,
Crandall and Rabinowitz \cite{cr75} first proved the boundedness of
 weak stable solutions   $u\in W^{1,2}_0(\Omega)$ to
 the equation \eqref{sta} with the nonlinearity $f(t)= e^t$, that is,
 the equation $-\Delta u= e^u$.   Later, Brezis and  V\'azquez \cite{bv97}  obtained an analogue result
 for  the equation \eqref{sta} with
the nonlinearity  $f(t)=(1+t)^p$ for any $p>1$, that is,    $-\Delta u= (1+u)^p$.
However, in dimension $n\ge10$, one cannot expect the boundedness of $W^{1,2}_0(\Omega)$
stable solutions  to \eqref{sta}  in general; indeed,
$-2\ln|x|\in W^{1,2}_0(B_1)\setminus L^\fz(B_1)$  is a stable solution to $-\bdz u=2(n-2)e^u$  in the unit ball $B_1$,
see \cite{bv97,b03} for details.

The results above lead to a long-standing conjecture  (see \cite{b03} and also \cite{cfrs}):
 in dimension $n\le 9$, if the nonlinearity $f$ is positive, nondecreasing, convex and superlinear at
$+\fz$, then stable solutions $u\in W^{1,2}_0(\Omega)$ to the equation \eqref{sta} are  always bounded.
This conjecture was
 also motivated by the Brezis' boundedness conjecture and Brezis-V\'azquez' $W^{1,2}$-regularity conjecture  for extremal solutions
(apriori in $L^1(\Omega)$) to the equation \eqref{sta}; for more details see \cite{b03}.
Recently, this conjecture
 was completely answered by  Cabr\'e, Figalli, Ros-Oton and Serra \cite{cfrs}; see also \cite{cc06,c10,cr13,csr16,c19,n20,v13} for earlier contributions towards this conjecture.
Indeed, if  $f$ is nonnegative, Cabr\'e, Figalli, Ros-Oton and Serra \cite{cfrs}  established
a crucial apriori interior  H\"older estimate for  $C^2$-stable solutions $u$ to \eqref{sta}:
\begin{align}\label{holder}
\underset{{B_r(x_0)}}{\rm osc}\ u\le C(n)\left(\frac{r}R\right)^{\alpha}\|Du\|_{L^2(B_{2R}(x_0))},\quad
\forall x_0\in\Omega,\, 0<r<R<\frac1 4\dist(x_0,\partial\Omega),
\end{align}
for some dimensional exponent $\alpha\in(0,1)$. Here
  $B_r(y)$ denotes the ball with center $y$ and radius
$r>0$. Moreover, they showed that   $L^{2+\gamma}(B_{r}(x_0))$-norm of $Du$  for some dimensional exponent $\gz>0$
and $C^{\alpha}(\overline {B_r(x_0)})$-norm of  $ u$   are
bounded by the $L^1(B_{2r}(x_0))$ norm of $u$.
Combining with moving plane method, they bounded $L^\fz(\Omega)$-norm of $u$ by  $L^1(\Omega)$-norm
when  $\Omega$ is convex $C^1$ domain and $u$ has zero boundary.
If $f$ is nondecreasing and convex   in addition,
they  bounded
 $L^{2+\gamma}(\Omega)$-norm of $Du$
and $C^{\alpha}(\overline {\Omega})$-norm of  $ u$ via  $L^1(\Omega)$-norm
when  $\Omega$ is of $ C^3$ and $u$ has zero boundary.
 Via an approximation argument they further bound the $L^\fz(\Omega)$-norm of all $W^{1,2}_0(\Omega)$-stable solutions via some constant depending only on $ \Omega $ and  rate $f(t)/t$,
where  they require $f(t)/t\to +\fz$  as $t \to+\fz$.
The above bound via $ L^1$-norm is very important for them
 to show that extremal solutions are indeed  $W^{1,2}_0$-weak solutions.
Note that their proofs, especially the proof of their Lemma 3.1 which is crucial to get \eqref{holder}, heavily relies on the nonnegativity of $ f$.

It is natural   to ask, in dimension $n\le 9$,  whether the nonnegativity of the nonlinearity $f$ is necessary to get the boundedness
of $W^{1,2}_0(\Omega)$-stable solutions to \eqref{sta}?
Cabr\'e \cite{c21} further asked
that  whether the nonnegativity of the nonlinearity $f$ is necessary to get the interior regularity
of   stable solutions?
It was already proved by  Cabr\'e and Capella \cite{cc06} that
the nonnegativity of   $f$ is  unnecessary to get the boundedness of
 radial stable solutions $u\in W^{1,2}_0(\Omega)$ to \eqref{sta}.
 Moreover, in dimension $n\le4$,  the nonnegativity of  $f$ is shown by Cabr\'e  \cite{c10} to be unnecessary to get the apriori boundedness regularity for
 $C^2$-stable solutions  to \eqref{sta}; but in dimension $5\le n\le 9$,
 his approach   does not work. Recently, in dimension $n\le4$,
  Cabr\'e \cite{c19} also provided a new and  simpler proof of this result in \cite{c10}. In dimension $5$, Cabr\'e \cite{c19} established the interior $L^p$ bounds of $C^2$-stable solutions with $p<+\fz$ for any $C^1$ nonlinearity $f$.

In this paper we prove that,   in dimension $n\le 5$,
 the nonnegativity of   $f$ is unnecessary to get the  following apriori
 interior H\"older regularity of $C^2(\Omega)$-stable  solutions to \eqref{sta}.

\begin{thm}\label{th-2}Let $2\le n\le 5$.
There exists $\alpha\in(0,1)$ and $C\ge1$ depending only on $n$ such that
   $C^2(\Omega)$-stable
solutions to equation \eqref{sta}  with $f\in C^{0,1}(\rr)$    enjoy  the interior H\"older $C^{0,\alpha}(\Omega)$-regularity  \eqref{holder}.
\end{thm}
Moreover, under   nondecrease and
convexity assumptions on $f$, we show that
the nonnegativity of   $f$  is   unnecessary to get
 the interior H\"older regularity of
$W^{1,2}_\loc(\Omega)$-stable solutions to \eqref{sta}.
For convenience, we  write
\begin{align*}
\mathcal{C}:=\{f:\rr\to \rr,\quad f\in C^{0,1}(\rr)\ {\rm is\ nondecreasing\ and\
convex}\}.
\end{align*}
We emphasis again that  the nonlinearity in the class $\mathcal C$  is not
necessarily nonnegative.

  \begin{thm}\label{th-1}
Let $2\le n\le 5$.
There exists $\alpha=\alpha(n)\in(0,1)$  and $C=C(n)\ge1$ such that
   $W^{1,2}_\loc(\Omega)$-stable
solutions to the equation \eqref{sta}  with $f\in\mathcal C $    enjoy  the   H\"older $C^{0,\alpha}(\Omega)$-regularity  \eqref{holder}
with  $0<r<R/2<r_0$ for some $r_0=r_0(f,n,x_0)>0$, and hence
are locally bounded in $\Omega$.
\end{thm}

Obviously,   we would conclude  Theorem \ref{th-1}  from
Theorem \ref{th-2}  if    $W^{1,2}_\loc(\Omega)$-stable solutions can be approximated via
 $C^2$-stable solutions in  certain sense.
Recall that if  $ f\in \mathcal{C}$ is nonnegative, it was proved in \cite[Proposition 4.2]{cfrs} that
$W^{1,2}_0$-stable solutions to \eqref{sta} is  globally approximated in $W^{1,2}(\Omega)$ by
 $ C^2(\Omega)$-stable solution  $u^\ez$ to the equation $-\bdz v=f_\ez(v)$ in $\Omega$ for some $0\le f_\ez\in \mathcal C$.
Below we  show  that,  in any dimension,
 the nonnegativity of $f$ is unnecessary to get  the following local approximation
 to $W^{1,2}_\loc(\Omega)$-stable solutions via  $C^2$-stable solutions.

\begin{thm}\label{app} Let $n\ge2$.
Let $u$ be a $W^{1,2}_\loc(\Omega)$-stable solution  to the equation \eqref{sta}, where $f\in \mathcal C$.
For any $x_0\in \Omega$  there exists $r_0>0$ depending on $f,n$ and $x_0$ so that
we can find a family   $\{u^\ez\}_{\ez\in(0,\ez_0)}\subset C^2(B_{r_0}(x_0)) $  for some sufficiently small $\ez_0>0$   satisfying that
$u^\ez$ is a stable solution to $-\Delta v=f_\ez(v)$ in $B_{r_0}(x_0)$ for some $f_\ez\in \mathcal C$,
and that $u^\ez\to u$ in $W^{1,2}(B_{r_0}(x_0))$ as $\ez\to 0$.
\end{thm}

We prove Theorem \ref{th-1} as below.

\begin{proof}[Proof of Theorem \ref{th-1}]
 Let $u$ be a $W^{1,2}_\loc(\Omega)$-stable solution  to the equation \eqref{sta}, where $f\in \mathcal C$.
For any $x_0\in \Omega$, let $r_0 >0$   and $ u_\ez\in C^2(B_{r_0}(x_0))$ be as in Theorem \ref{app}.
Applying Theorem \ref{th-2} to $u_\ez$, one has
\begin{align}\label{holder1}{\rm osc}_{B_r(x_0)}u_{\ez}\le C(n )\left(\frac{r}{r_0}\right)^{\az}\|Du_{\ez}\|_{L^2(B_{r_0}(x_0))},\quad \forall r<r_0/2.
\end{align}
By Theorem \ref{app}, $u_{\ez}\to u$ in $W^{1,2}(B_{r_0}(x_0))$ as $\ez\to 0$, one  concludes that
 \eqref{holder1} also holds for $u$.  We complete this proof.
\end{proof}

The proof of Theorem \ref{th-2} and Theorem \ref{app} is presented in Section 2 and Section 3 respectively. Below we sketch some ideas of the proofs.

To prove  Theorem \ref{th-2}, it suffices to built up the following crucial inequality for $C^2$-stable solutions $u$ to   \eqref{sta}:
\begin{align}\label{m-1}
\int_{B_r}|x|^{-n+2}|Du|^2\,dx\le C(n)\int_{B_{2r}\backslash B_{r}}|Du|^2|x|^{-n+2}\,dx,
\end{align}
that is, Lemma 2.1 in Section 2. We write $B_r=B_r(0)$  and assume $0\in\Omega$ for simple.
Then a hole-filling together with an iteration argument
 leads to Theorem \ref{th-2}, see Section 2 for details.

In order to prove \eqref{m-1} or Lemma \ref{key-es},  we choose test function
$ \xi=|Du||x|\eta$ in \eqref{stt}. Equivalently,  we
choose test functions $\zeta=|x|\eta$ for any $\eta\in C^{0,1}_c(\Omega)$ in
 the reversed
Poincar\'e type formula in Lemma \ref{lpo-f}, which was proved by  \cite{sz98} (see also \cite{cfrs,sz}), so to obtain
  Lemma \ref{le-1}.
Our test functions are different from   test functions $\xi=(x\cdot Du)\eta$
  chosen in \eqref{stt} by \cite[Lemma 2.1]{cfrs}, but it is similar to test functions $|Du||x|^{-\bz}$ with some $\bz>0$ chosen by Cabr\'e \cite[Theorem 1.5]{c19}.
Then we have to bound all  terms including Hessian (or second order derivatives) in Lemma \ref{le-1}. To this end,
 we decompose $D^2u Du$ along the normal direction and the tangential space of  the level set $\{u=c\} $ for any $c$, that is,
$$\frac{D^2uDu}{|Du|}
 = \frac{\Delta_\fz u}{|Du|^3}  Du +
[\frac{D^2u Du}{|Du|}-\frac{\bdz_\fz u}{|Du|^3}Du]\quad{\rm on}\quad \{Du\neq 0\};$$
where $\Delta_\fz u=D^2uDu\cdot Du$  is the $\infty$-Laplacian. In order to bound $ \frac{\Delta_\fz u}{|Du|^3}  Du$, we further write
$$\frac{\Delta_\fz u}{|Du|^3}  Du= \left(\frac{\Delta_\fz u}{|Du|^2 } -\Delta u\right) \frac{Du}{|Du|} +\Delta  u \frac{Du}{|Du|}\quad{\rm on}\quad \{Du\neq0\}.$$
Then an upper bound for $(\frac{\Delta_\fz u}{|Du|^3 } -\frac{\Delta u}{|Du|}) $,
the mean curvature of $\{u=c\}$ up
to multiple $(n-1)$, was  derived in Lemma \ref{l-geo} by using the Cauchy-Schwartz inequality.
Via integration by parts we also have an  upper bound for the integration of $\int_\Omega \Delta  u  (Du\cdot x)\eta^2\,dx$ as in lemma \ref{le-3}.
Combining all of them we could get an estimate without Hessian terms; see Lemma \ref{le-4}.
 Taking $\eta =|x|^{-\frac{n-2}2}\phi$ we prove  Lemma \ref{key-es} via a direct calculation, that is, \eqref{m-1}.

To prove Theorem 1.3, we first built up  in Lemma 3.1 the existence of $C^2$-solution to the equation
$-\Delta v =f'_{-}(1)v-[f'_{-}(1)-f(1)]$ in  balls  $B_{r_0}(x_0)$ with given Dirichlet boundary $u$,
where  the range of   $r_0$ is determined by the Sobolev inequality and $\dist(x_0,\partial\Omega)$.  Thanks to Lemma 3.1, we are able to adapt the arguments for
\cite[Proposition 4.2]{cfrs} to
 approximate
 locally the stable solution via  $C^2$-stable solutions.
Details of the proof of Theorem 1.3  are given in Section 3 for reader's convenience.

Finally, we list several remarks.
\begin{rem}\rm

(i)
 The $W^{1,2}_\loc(\Omega)$-stable solutions  in
Theorem \ref{th-1} can not be relaxed to some $L^1_\loc(\Omega)$-stable solutions.
Indeed, in dimension $ 3\le n\le 9$,
  consider the equation  $-\bdz u=\lambda_s(1+u)^p$ in $B_1$, where
$$ \frac{n}{n-2}<p\le \frac{n+2\sqrt {n-1}}{n-4+2\sqrt {n-1}},\quad
\lambda_s=\frac{2}{p-1}\left(n-\frac{2p}{p-1}\right)>0,$$
The function $ |x|^{-\frac{2}{p-1}}-1\in L^1(B_1)  $
satisfies \eqref{stt00} and \eqref{stt}. But
$ |x|^{-\frac{2}{p-1}}-1\notin W^{1,2}_\loc (B_1)$
and
$ |x|^{-\frac{2}{p-1}}-1\notin  L^{\fz}_\loc (B_1)  $, see
 Brezis and V\'azquez \cite{bv97} for more details.

(ii)   In Theorem \ref{th-2} and also Theorem \ref{th-1}, it is not clear to us whether the $L^2 (B_r)$-norm  of $Du$ can be bounded by  $L^1(B_{2r})$ norm of $u$.
Recall that in the case  $f\ge0$,  this is proved in \cite{cfrs}  and
is crucial  to get   $W^{1,2}$-regularity of extremal
 solutions.

(iii) In dimension $n\le 5$, it remains open whether the nonnegativity of $f$ is unnecessary to  get
the  global  aproiori H\"older regularity of $C^2(\Omega)$-stable solutions with zero boundary, global approximation to $W^{1,2}_0(\Omega)$-stable solutions via $C^2(\Omega)$-stable solutions
  and also global  H\"older regularity of $W^{1,2}_0(\Omega)$-stable solution.
\end{rem}

\section{Proof of Theorem \ref{th-2}}

To prove Theorem \ref{th-2} it suffices to prove the following key lemma. In  this section, we always let   $u\in C^2(\Omega)$ be any stable solution   to the equation \eqref{sta}, where $f\in C^{0,1}(\rr)$.

\begin{lem}\label{key-es} Assume that $0\in\Omega$.
We have
\begin{align}\label{ke}
&\frac{(n-2)(6-n)}{4}\int_{\Omega}|Du|^2|x|^{-n+2}\phi^2\,dx
+(n-3)\int_{\Omega}(Du\cdot x)^2|x|^{-n}\phi^2\,dx\nonumber\\
&\quad\quad\le \int_{\Omega} |Du|^2|x|^{-n+4}|D\phi|^2\,dx
-n\int_{\Omega}|Du|^2|x|^{-n+2}(x\cdot D\phi) \phi\,dx\nonumber\\
&\quad\quad\quad +4\int_{\Omega}(x\cdot Du)(Du\cdot D\phi) \phi|x|^{-n+2}\,dx \quad\forall \phi\in C^\fz_c(\Omega).
\end{align}
\end{lem}

\begin{proof}[Proof of Theorem \ref{th-2}]
We may assume $3\le n\le 5$. Indeed, in dimension 2,
one may add  an additional artificial variable similarly to  \cite[Section 1.4]{cfrs} and   \cite[Remark 1.7]{cms}.
Let $u\in C^2(\Omega)$ be a stable solution to \eqref{sta}.
 Below we only need to show that
 \eqref{holder} holds
 for some $\alpha\in(0,1)$ and constant $C$ depending only on $n$.
 Let $x_0\in \Omega$ and $0<r<R\le \frac14\dist(x_0,\partial\Omega)$.
Up to  a translation argument, we may  $x_0=0$. Since $3\le n\le 5$, we have the coefficient $\frac{(n-2)(6-n)}{4}>0$ in \eqref{ke}.
Applying Young's inequality, from \eqref{ke} we deduce that
$$
  \int_{\Omega}|Du|^2|x|^{-n+2}\phi^2\,dx \le C\int_{\Omega} |Du|^2|x|^{-n+4}|D\phi|^2\,dx  \quad\forall \phi\in C^\fz_c(\Omega) .
$$
Given any $k\ge-1$,   taking the cut-off function
$\phi\in C^\fz_c(B_{2^{-k}R})$ satisfying
\begin{align*}
\phi=1\quad{\rm in}\ B_{2^{-k-1}R},\quad 0\le\phi\le 1\quad
{\rm in}\ B_{2^{-k}R},\quad |D\phi|\le \frac 4{2^{-k}R}
\quad{\rm in}\ B_{2^{-k}R},
\end{align*}
  we get
\begin{align}\label{th1-1}
\int_{B_{2^{-k-1}R}}|Du|^2|x|^{-n+2}\,dx
\le C(n)\int_{B_{2^{-k}R}\backslash B_{2^{-k-1}R}}|Du|^2|x|^{-n+2}\,dx.
\end{align}
Adding both sides by $C(n)\int_{B_{2^{-k-1}R}}|Du|^2|x|^{-n+2}\,dx$ one has
\begin{align*}
\int_{B_{2^{-k-1}R}}|Du|^2|x|^{-n+2}\,dx
\le \frac{C(n)}{C(n)+1}\int_{B_{2^{-k}R}}|Du|^2|x|^{-n+2}\,dx\quad \forall k\ge0.
\end{align*}
By iteration, we further have
\begin{align*}
\int_{B_{2^{-k-1}R}}|Du|^2|x|^{-n+2}\,dx
\le \left[\frac{C(n)}{C(n)+1}\right]^k\int_{B_{ R}}|Du|^2|x|^{-n+2}\,dx
\quad \forall k\ge0.
\end{align*}
Writing $\alpha=\log_2 \frac{C(n)}{C(n)+1}$, and applying \eqref{th1-1} with $k=-1$ we further have
\begin{align*}
\int_{B_{2^{-k-1}R}}|Du|^2|x|^{-n+2}\,dx \le
  C(n)2^{k\alpha} \int_{B_{ 2R}}|Du|^2 \,dx
\quad \forall k\ge0.
\end{align*}
Letting $k\ge 1$ such that $2^{-k-2}R\le r\le 2^{-k-1}R$, from this we conclude \eqref{holder} as desired.
\end{proof}

The rest of this section is devoted to the proof of Lemma \ref{key-es}, which consists of a sequence of lemmas.
We begin with   the following important reversed Poincar\'e type formula
due to Sternberg and Zumbrun \cite{sz98,sz}; see also Cabr\'e et al \cite[Lemma 2.3]{cfrs}
for an alternative proof which will be stated here
for the convenience of reader.

\begin{lem}\label{lpo-f}
 We have
\begin{align}\label{po-f}
\int_\Omega [|D^2u|^2-|D|Du||^2]\zeta^2\le \int_\Omega |Du|^2|D\zeta|^2\,dx\quad\forall\zeta\in C^{0,1}_c(\Omega).
\end{align}
\end{lem}
\begin{proof}
Recall the following divergence structure
$$|D^2v|^2-(\bdz v)^2={\rm div}(
D^2v Dv - \bdz v  Dv)\quad\forall v\in C^3.$$
By an approximation argument we have
\begin{align*}
\int_\Omega [|D^2u|^2-(\bdz u)^2]\zeta^2\,dx=2\int_\Omega [\bdz u (Du\cdot D\zeta)
-D^2u Du\cdot D\zeta ]\zeta\,dx\quad \forall\zeta\in C^{0,1}_c(\Omega).
\end{align*}
Since $-\bdz u=f(u)$,  via integration by parts we obtain
\begin{align*}
\int_\Omega (\bdz u)^2\zeta^2\,dx&=-\int_\Omega (\bdz u) f(u)\zeta^2\,dx\nonumber\\
&=\int_\Omega f'_{-}(u)|Du|^2\zeta^2\,dx+2\int_\Omega f(u)(Du\cdot D\zeta)\zeta\,dx\nonumber\\
&=\int_\Omega f'_{-}(u)|Du|^2\zeta^2\,dx-2\int_\Omega \bdz u(Du\cdot D\zeta)\zeta\,dx.
\end{align*}
By the definition of the stable solution,   one has
\begin{align*}
 \int_\Omega f'_{-}(u)|Du|^2\zeta^2\,dx&\le
\int_\Omega |D(|Du|\zeta)|^2\,dx\nonumber\\
&=\int_\Omega |D|Du||^2\zeta^2\,dx+\int_\Omega |Du|^2|D\zeta|^2\,dx
+2\int_\Omega (D^2u Du\cdot D\zeta)\zeta\,dx.
\end{align*}
Combing all these estimates together we conclude \eqref{po-f}.
\end{proof}

Taking $\zeta=|x|\eta$ in \eqref{po-f}  with
  $\eta\in C^{0,1}_c(\Omega)$ we obtain the following.

\begin{lem}\label{le-1}
 We have
  \begin{align}\label{d1-1}
 (n-1)\int_\Omega |Du|^2\eta^2\,dx
 &\le -\int_\Omega [|D^2u|^2-|D|Du||^2]|x|^2\eta^2\,dx-2\int_\Omega (D^2u Du\cdot x)\eta^2\,dx\nonumber \\
&\quad+\int_\Omega |Du|^2|D\eta|^2|x|^2\,dx
\quad\forall \eta\in C^{0,1}_c(\Omega).
 \end{align}
\end{lem}
\begin{proof}
 Choose $\zeta=|x|\eta$ with $\eta\in C^{0,1}_c(\Omega)$. Clearly,
$\zeta\in C^{0,1}_c(\Omega)$.
By \eqref{po-f} we have
$$
 \int_\Omega [|D^2u|^2-|D|Du||^2]|x|^2\eta^2\,dx\le \int_\Omega |Du|^2|D[|x|\eta]|^2\,dx.
$$
A direct computation gives
\begin{align*}
  \int_\Omega |Du|^2|D(|x|\eta)|^2\,dx
&=\int_\Omega |Du|^2|D\eta|^2|x|^2\,dx+\int_\Omega |Du|^2\eta^2\,dx
+2\int_\Omega |Du|^2 (D\eta \cdot x) \eta\,dx.
\end{align*}
By integration by parts   we have
\begin{align}\label{d1-1x}
2\int_\Omega |Du|^2 (D\eta \cdot x) \eta\,dx&=-\int_\Omega \eta^2{\rm div}(|Du|^2x)\,dx\nonumber\\
&=-n\int_\Omega |Du|^2\eta^2\,dx-2\int_\Omega (D^2u Du\cdot x)\eta^2\,dx.
\end{align}
Thus  \eqref{d1-1} holds.
\end{proof}

 To bound the second order derivatives of $u$ in \eqref{d1-1}, we decompose $D |Du|=\frac{D^2uDu}{|Du|}$ as the summation of
its projection  on the normal direction  of the level set  $\{u=c\}$ whenever $Du\ne 0$,
and also   its projection    on the tangential space   of  $\{u=c\}$.
That is, when $Du\ne0$,
\begin{align}\label{d1-2}\frac{D^2uDu}{|Du|}&= \frac{D^2uDu}{|Du|}\cdot\frac{Du}{|Du|}\frac{Du}{|Du|}+ \left(\frac{D^2uDu}{|Du|}-\frac{D^2uDu}{|Du|}\cdot \frac{Du}{|Du|} \frac{Du}{|Du|}\right) \nonumber\\
&= \frac{\Delta_\fz u}{|Du|^3}  Du +
\left(\frac{D^2u Du}{|Du|}-\frac{\bdz_\fz u}{|Du|^3}Du\right)\quad {\rm on}\ \{Du\neq 0\},\end{align}
  where and below we write $\Delta_\fz u=D^2uDu\cdot Du$.
  The following bound follows from \eqref{d1-2}.

\begin{lem}\label{le-2}
One has
\begin{align}\label{d1-3}
 -2D^2uDu\cdot x
&\le  \frac{1}{n-1} \Big(\bdz u-\frac{\bdz_\fz u}{|Du|^2}\Big)^2 |x|^2
+ \left[|D|Du||^2-\left(
\frac{\bdz_\fz u}{|Du|^2}\right)^2\right]|x|^{2} \nonumber\\
&\quad- 2  \bdz u (Du\cdot x)
 + (n-1) (Du\cdot x)^2 |x|^{-2}   + |Du|^2\quad  a.\,e.\ \mbox{in $\Omega$}.
\end{align}

\end{lem}

\begin{proof}
Observe that, since $u\in C^2$, we know that $D^2u=0$ almost everywhere in $\{Du=0\}$;
 see \cite[Theorem 1.56]{t87}.  So \eqref{d1-3} holds almost everywhere in $\{Du=0\}$.
Below we assume that $Du\ne 0$. Considering \eqref{d1-2} and by $D|Du|=\frac{D^2uDu}{|Du|}$ we write
\begin{align*}
2[\bdz u (Du\cdot x)-D^2uDu\cdot x ]
&=2\left(\bdz u\frac{Du\cdot x}{|Du|} - D|Du|\cdot x \right)|Du| \nonumber\\
&=  2\left( {\bdz u} - \frac{\Delta_\fz u}{|Du|^2}\right) (Du\cdot x) - 2 \left( D|Du|- \frac{\Delta_\fz u}{|Du|^3} Du \right)\cdot x  |Du|.
\end{align*}
By Young's inequality  we get
\begin{align*}
2\left(\bdz u - \frac{\Delta_\fz u}{|Du|^2}\right) (Du\cdot x)
&\le \frac{1}{n-1} \Big(\bdz u-\frac{\bdz_\fz u}{|Du|^2}\Big)^2 |x|^2
+(n-1) (Du\cdot x)^2 |x|^{-2}
\end{align*}
and
$$-
2\left( D|Du|- \frac{\Delta_\fz u}{|Du|^3} Du \right)  \cdot x |Du|\le
\left| D|Du|- \frac{\Delta_\fz u}{|Du|^3} Du \right|^2|x|^2+  |Du|^2.$$
Since  $D|Du|=\frac{D^2u Du}{|Du|}$,  one has \begin{align*}
 \left| D|Du|-\frac{\bdz_\fz u}{|Du|^3}Du\right|^2
&= \left|\frac{D^2u Du}{|Du|}\right|^2-2\frac{D^2u Du\cdot Du}{|Du|} \frac{\bdz_\fz u}{|Du|^3}  +\left|\frac{\bdz_\fz u}{|Du|^3}Du\right|^2\\
&=  |D|Du||^2- \left(\frac{\bdz_\fz u}{|Du|^2}\right)^2.
\end{align*}
Combining all together we get \eqref{d1-3}.
\end{proof}

Observing that
$$
-\frac1{n-1}\left(\frac{\bdz u}{|Du|}-\frac{\bdz_\fz u}{|Du|^3}\right)
=-\frac1{n-1}{\rm div}\left(\frac{Du}{|Du|}\right)$$ is the mean curvature of the level set $\{u=c\}$,
one has the following upper bound.

\begin{lem}\label{l-geo} One has
\begin{align}\label{geo}
\left(\bdz u-\frac{\bdz_\fz u}{|Du|^2}\right)^2
\le (n-1)
\left[|D^2u|^2-2|D|Du||^2+\left(\frac{\bdz_\fz u}{|Du|^2}\right)^2\right]\quad {\rm a.e.\
in}\quad \Omega.
\end{align}
\end{lem}

\begin{proof}
Recall that  $u\in C^2(\Omega)$ implies $D^2u=0$ a.e. in $\{Du=0\}$; see \cite[Theorem 1.56]{t87}. Then  \eqref{geo} always holds
a.e. in $\{Du=0\}$. Next we only consider the case of $\{Du\neq 0\}$.
Now we may assume that $Du\neq 0$. The mean curvature $H_u$ of
level set $\{u=c\}$  satisfies
$$-(n-1)H_u={\rm div}\left(\frac{Du}{|Du|}\right),$$
for  detail see \cite[Proposition 3.2]{cs}.
Let $\lambda_1,...,\lambda_{n-1}$ be
principal curvatures of the level set of $u$. By the definition of mean curvature $H_u$,
$$H_u=\frac{\sum_{i=1}^{n-1}\lambda_i}{n-1}.$$
By Cauchy-Schwartz inequality one has
$$\left({\rm div}\left(\frac{Du}{|Du|}\right)\right)^2=\left(\sum_{i=1}^{n-1}\lambda_i\right)^2\le (n-1)\sum_{i=1}^{n-1}(\lambda_i)^2.$$
We multiply both sides by $|Du|^2$ and apply
$${\rm div}\left(\frac{Du}{|Du|}\right)=\frac{\bdz u}{|Du|}-\frac{\bdz_\fz u}{|Du|^3}$$
to conclude
$$\left(\bdz u-\frac{\bdz_\fz u}{|Du|^2}\right)^2\le (n-1)|Du|^2\sum_{i=1}^{n-1}(\lambda_i)^2.$$
Recall the following geometric identity
$$|Du|^2\sum_{i=1}^{n-1}(\lambda_i)^2=|D^2u|^2-2|D|Du||^2+
\left(\frac{\bdz_\fz u}{|Du|^2}\right)^2,$$
see \cite[Lemma 2.1]{sz}. We get the desired inequality \eqref{geo}.
\end{proof}

\begin{rem}\rm
In terms of matrix language, the proof of above inequality is much more direct but less geometric information involved.
Given any $n\times n$ symmetric matrix $M$, denote by $M_{n-1}$  the  $n-1$ order principal sub-matrix of $M$.
Then
$$|M|^2=|M_{n-1}|^2+2|Me_n|^2-|\langle Me_n,e_n\rangle|^2.  $$
Since
$$|M_{n-1}|^2\ge \frac{({\rm tr}M_{n-1})^2}{n-1}= \frac{({\rm tr}M -\langle Me_n,e_n\rangle)^2}{n-1}.$$
Thus
$$({\rm tr}M -\langle Me_n,e_n\rangle)^2 \le (n-1) [|M|^2-2|Me_n|^2+
|\langle Me_n,e_n\rangle|^2 ].$$
Up to some rotation, this also holds with $e_n$   replaced by any unit vector $\xi$.
Applying this to $M=D^2u$ and $\xi=\frac{Du}{|Du|}$ one gets \eqref{geo} when $Du\ne0$.
\end{rem}

\begin{lem}\label{le-3}
 For all $\eta\in C^{0,1}_c(\Omega)$, we have
 \begin{align} \label{d1.xx}
& -2\int_\Omega \bdz u (Du\cdot x)\eta^2\,dx\nonumber\\
&\quad= (2-n)\int_\Omega  |Du|^2\eta^2\,dx-2\int_\Omega |Du|^2 (D\eta \cdot x) \eta\,dx +4\int_\Omega   (Du\cdot x)(Du\cdot D\eta)\eta \,dx.
\end{align}

\end{lem}
\begin{proof}
By integration by parts, we have
\begin{align*}
 -2\int_\Omega \bdz u (Du\cdot x)\eta^2\,dx&= 2\int_\Omega  Du\cdot D[(Du\cdot x)\eta^2]
\,dx\nonumber\\
&= 2\int_\Omega  D^2uDu\cdot x  \eta^2\,dx + 2\int_\Omega  |Du|^2\eta^2\,dx\nonumber\\
&\quad +4\int_\Omega   (Du\cdot x)(Du\cdot D\eta)\eta \,dx.
\end{align*}
Since \eqref{d1-1x} gives
\begin{align}
 2\int_\Omega (D^2u Du\cdot x)\eta^2\,dx= -2\int_\Omega |Du|^2 (D\eta \cdot x) \eta\,dx -
 n\int_\Omega |Du|^2\eta^2\,dx, \nonumber
\end{align}
one has \eqref{d1.xx}.
\end{proof}

\begin{lem}\label{le-4}We have
 \begin{align}\label{d1-5}
 2 ( n-2)\int_\Omega |Du|^2\eta^2\,dx
 &\le  (n-1)\int_\Omega   (Du\cdot x)^2 |x|^{-2}    \eta^2\,dx   \nonumber \\
&\quad+\int_\Omega |Du|^2|D\eta|^2|x|^2\,dx-2\int_\Omega |Du|^2 (D\eta \cdot x) \eta\,dx\nonumber\\
&\quad +4\int_\Omega   (Du\cdot x)(Du\cdot D\eta)\eta \,dx
\quad\forall \eta\in C^{0,1}_c(\Omega).
 \end{align}
\end{lem}
\begin{proof}

From \eqref{d1-3} and \eqref{geo}, we deduce
\begin{align*}
&-2D^2uDu\cdot x   \nonumber\\
&\le \left[|D^2u|^2-2|D|Du||^2+\left(\frac{\bdz_\fz u}{|Du|^2}\right)^2\right]|x|^2
+ \left[|D|Du||^2-\left(
\frac{\bdz_\fz u}{|Du|^2}\right)^2\right]|x|^{2} \\
&\quad -2  \bdz u (Du\cdot x)+ (n-1) (Du\cdot x)^2 |x|^{-2}   + |Du|^2  \\
&=  [|D^2u|^2- |D|Du||^2]|x|^2-2  \bdz u (Du\cdot x)
 + (n-1) (Du\cdot x)^2 |x|^{-2}   + |Du|^2 .
\end{align*}
Plugging this into \eqref{d1-1}, noting \eqref{d1.xx} we conclude
\eqref{d1-5}.
\end{proof}

We utilize this lemma to  prove Lemma \ref{key-es} as below.

\begin{proof}[Proof of Lemma \ref{key-es}]
Since $u\in C^2(\Omega)$ and $0\in\Omega$, up to some approximation argument
(see \cite[Lemma 2.1]{cfrs}),
\eqref{d1-5} also holds for $\eta=|x|^{-\frac{n-2}{2}}\phi$ with $\phi\in C^\fz_c(\Omega)$.
Then applying \eqref{d1-5} with $\eta=|x|^{-\frac{n-2}{2}}\phi$, we have
 \begin{align}\label{d1-6}
& 2( n-2)\int_\Omega |Du|^2|x|^{-n+2}\phi^2\,dx\nonumber\\
 &\le  (n-1)\int_\Omega   (Du\cdot x)^2    |x|^{-n }\phi^2\,dx \nonumber \\
&\quad+\int_\Omega |Du|^2|D[|x|^{-\frac{n-2}2}\phi ]|^2|x|^2\,dx-2\int_\Omega |Du|^2 (D[|x|^{-\frac{n-2}2}\phi]\cdot x)|x|^{-\frac{n-2}2}\phi\,dx\nonumber\\
&\quad +4\int_\Omega   (Du\cdot x)(Du\cdot D[|x|^{-\frac{n-2}2}\phi])|x|^{-\frac{n-2}2}\phi \,dx   \nonumber\\
&:=J_1+J_2+J_3+J_4.
 \end{align}
Note that
$$
D[|x|^{-\frac{n-2}{2}}\phi]=-\frac{n-2}{2}|x|^{-\frac{n-2}{2}-2}x\phi+|x|^{-\frac{n-2}{2}}D\phi.
$$
 A direct calculation  yields
\begin{align*}
J_1+J_4&=[n-1-2(n-2) ]\int_\Omega  (x\cdot Du)^2|x|^{-n}\phi^2\,dx +4\int_\Omega (x\cdot Du)(Du\cdot D\phi) \phi|x|^{-n+2}\,dx,
\end{align*}
\begin{align*}
J_2&=\frac{(n-2)^2}{4}\int_\Omega  |Du|^2|x|^{-n+2}\phi^2\,dx
+\int_\Omega  |Du|^2|x|^{-n+4}|D\phi|^2\,dx\nonumber\\
&\quad-(n-2)\int_\Omega |Du|^2|x|^{-n+2}(x\cdot D\phi) \phi\,dx
\end{align*}
and
\begin{align*}
 J_3&= (n-2)\int_\Omega  |Du|^2|x|^{-n+2}\phi^2 \,dx-
2\int_\Omega |Du|^2(x\cdot D\phi) \phi|x|^{-n+2}\,dx.
\end{align*}

We move the term
$$[n-1-2(n-2) ]\int_\Omega  (x\cdot Du)^2|x|^{-n}\phi^2\,dx= -(n-3)\int_\Omega  (x\cdot Du)^2|x|^{-n}\phi^2\,dx$$
in $J_1+J_4$ to the left-hand side of \eqref{d1-6}, and   also move the terms
$$\frac{(n-2)^2}{4}\int_\Omega  |Du|^2|x|^{-n+2}\phi^2\,dx\quad\mbox{and}\quad (n-2)\int_\Omega  |Du|^2|x|^{-n+2}\phi^2 \,dx$$
in $J_2$ and $J_3$ to the left-hand side \eqref{d1-6} so to get
$$[2(n-2)- (n-2)-\frac{(n-2)^2}{4}]\int_\Omega  |Du|^2|x|^{-n+2}\phi^2 \,dx= \frac{(n-2)(6-n)}4 \int_\Omega  |Du|^2|x|^{-n+2}\phi^2 \,dx.$$
Then we conclude  \eqref{ke} from \eqref{d1-6} and $J_1+J_4$, $J_2$ and $J_3$.
 \end{proof}

\section{Proof of Theorem \ref{app}}

Given any $r>0$, denote by $\Sigma(r)$  the spectrum of Dirichlet  Laplacian in $B_{r}$, that is,
the collection of all $\lz>0$ so that one can find $0\ne v\in W^{1,2}_0(B_{r})$ satisfying
$
-\bdz v=\lz v
$ in $B_{r } $.
Write $\lz_1$ as the first eigenvalue, that is, $\lz_1>0$ is the minimal element
of $\Sigma (r)$.
Recall that Sobolev inequality
$$\int_{B_r}|v|^2\,dx\le C_0(n)r^2\int_{B_r}|Dv|^2\,dx\quad\forall v\in W^{1,2}_0(B_r)$$
for some constant   $ C_0(n)\ge 1$.  One has $
\lz_1\ge [C_0(n){r}^{2}]^{-1}.
$

Given any  $f\in \mathcal{C}$,  the nondecreasing property implies   $f'_-\ge 0$ in $\rr$.
Choose $r_\ast=r_\ast(f,n)>0$  such that
\begin{equation}\label{rfn} [C_0(n){r_\ast}^{ 2}]^{-1}>8[1+f'_-(1)].
\end{equation}
We have $f'_-(1)\notin \Sigma (r)$ for all $0<r<r_\ast$.

\begin{lem}\label{ex}
Let $u\in W^{1,2}_\loc(\Omega)$ be a solution to the equation
$-\bdz u=f(u)$ in $\Omega$, where $f\in \mathcal{C}$. For any point $x_0\in \Omega$,
there exists $0<r_0<\min\{r_\ast,\frac14\dist(x_0,\partial\Omega)\}$   such that
$u^{(0)}\in C^2(B(x_0,r_0))\cap W^{1,2}(B(x_0,r_0))$ is a unique solution  to
\begin{align}\label{inp-0}
\left\{
\begin{aligned}
-\bdz u^{(0)}&=f'_-(1)u^{(0)}-[f'_-(1)-f(1)]\quad&{\rm in}\ B(x_0,r_0),\\
u^{(0)}&=u\quad&{\rm on}\ \partial B(x_0,r_0),
\end{aligned}
\right.
\end{align}
and moreover,
\begin{align}\label{cp}
u^{(0)}\le u\quad{\rm a.e.\ in}\quad B(x_0,r_0).
\end{align}

\end{lem}

 \begin{proof}[Proof of Lemma \ref{ex}]
Without loss of generality, we  assume that $x_0=0$.
 Since $r_0<r_\ast$,  $f'_-(1)\notin \Sigma (r_0)$.
For simple we write
\begin{align}\label{ef}
A=f'_-(1)\ge0\quad{\rm and}\quad K=f'_-(1)-f(1).
\end{align}
To get $ u^{(0)} $ required by Lemma \ref{ex}
it suffices to  prove the existence, uniqueness and nonpositivity  of solutions $ w^0$ to \begin{align}\label{inp1-1}
-\bdz w   -Aw   =\bdz u +Au -K\quad{\rm in}\ B_{r_0};
w  =0\quad{\rm on}\ \partial B_{r_0}.
\end{align}
Indeed,  $w^0$ is the unique weak solution to \eqref{inp1-1} if and only if $u^{0}=w^0 +u$
is the unique weak
solution of problem \eqref{inp-0}, and moreover, \eqref{cp} is equivalent to $w^0\le0$.
We also note that by the standard elliptic theory (see \cite{gt98}), $u^{(0)} \in C^2(B_{r_0})$.

First we show the existence of weak solutions to \eqref{inp1-1}.
Since $-\bdz u=f(u)$ only
  belongs to $L^1(B_{r_0})$, we cannot use   the Fredhlom alternative theorem directly.
Instead, we consider the smooth mollification   $u^\dz:=u\ast \eta_{\dz}$ in $B_{2r_0}$, where $\dz\in(0,r_0/2]$  and
 $\eta_{\dz}$ is a standard
mollifier.   Since $A\notin \Sigma (r_0)$ and $ \bdz u^\dz\in L^2(B_{r_0})$,
by Fredhlom alternative theorem \cite[Chapter 6, Theorem 5]{e98},
there exists a unique
solution $w^{\dz}\in W^{1,2}_0(B_{r_0})$  such that
\begin{align}\label{inp1}
-\bdz w^\dz -Aw^\dz =\bdz u^\dz+Au^\delta-K\quad{\rm in}\ B_{r_0};
w^\dz =0\quad{\rm on}\ \partial B_{r_0}.
\end{align}
We claim that  $w^{\dz}\in W^{1,2}_0(B_{r_0})$ uniformly in $\dz\in(0,r_0/2]$.
Assume that this claim holds for the moment. By the weak compactness,  $w^\dz$ converges to some function
$w^0\in W^{1,2}_0(B_{r_0})$  weakly in $W^{1,2}(B_{r_0})$ and also in $L^2(B_{r_0})$, as
$\dz\to 0$.
Thanks to this, multiplying
\eqref{inp1} with $\phi\in W^{1,2}_0(B_{r_0})$, by integration by parts and then letting
$\dz\to 0$ we have
\begin{align}\label{beq}
\int_{B_{r_0}} Dw^0\cdot D\phi\,dx-A\int_{B_{r_0}} w^0\phi
=\int_{B_{r_0}} Du\cdot D\phi\,dx-\int_{B_{r_0}}(Au-K)\phi\,dx.
\end{align}
that  is,     $w^0$ solves \eqref{inp1-1}

To see the above claim that $w^{\dz}\in W^{1,2}_0(B_{r_0})$ uniformly in $\dz\in(0,r_0/2]$, multiplying
\eqref{inp1} with $w^{\dz}$ and integrating, via integration by parts and Young's inequality we have
\begin{align*}
\int_{B_{r_0}}|Dw^{\dz}|^2\,dx&= A\int_{B_{r_0}} |w^{\dz}|^2\,dx-
\int_{B_{r_0}} Dw^{\dz}\cdot Du{^\dz}\,dx+\int_{B_{r_0}}(Au^\delta -K)w^\dz\,dx \\
&\le (A+1)\int_{B_{r_0}} |w^{\dz}|^2\,dx+ \frac12\int_{B_{r_0}} |Dw^{\dz}|^2\,dx+ \frac12\int_{B_{r_0}} [|Du^{\dz}|^2+|Au^{\dz}-K|^2]\,dx.
\end{align*}
By the Sobolev inequality, one has
\begin{align*}
\int_{B_{r_0}}|Dw^{\dz}|^2\,dx
&\le 2C_0(n)[A+1]r^2_0\int_{B_{r_0}} |Dw^{\dz}|^2\,dx+    \int_{B_{r_0}}|Du^\dz|^2\,dx +
 \int_{B_{r_0}}|Au^{\dz}-K|^2\,dx.
\end{align*}
Since $C_0(n)[A+1]r^2_0\le \frac18$ and $u^\dz=u\ast \eta_{\dz}\in W^{1,2}(B_{r_0})$ uniformly in
$\dz\in (0,r_0/2]$, we have
 \begin{align*}
\int_{B_{r_0}}|Dw^{\dz}|^2\,dx
&\le     2 \int_{B_{2r_0}}|Du |^2\,dx +
 2\int_{B_{2r_0}}|Au-K|^2\,dx
\end{align*}
as desired.

Now we show that $ w^0$ is the unique solution to \eqref{inp1-1}.
Assume that  $ \wz w^0 \in W^{1,2}_0(B_{r_0})$ is also a solution  to   \eqref{inp1} different from $ w^0$.
Note that   \eqref{beq} also holds for $  \wz w^0$. Write $  w=w^0-\wz w^0\in W^{1,2}_0(B_{r_0})$. One has
\begin{align*}
\int_{B_{r_0}} D  [w^0-\wz w^0]\cdot D\phi\,dx-A\int_{B_{r_0}}  [w^0-\wz w^0]\phi
=0\quad \forall \phi\in W^{1,2}_0(B_{r_0}),
\end{align*}
that is,   in weak sense one has
$$-\bdz [w^0-\wz w^0]=A[w^0-\wz w^0]\quad{\rm in}\quad B_{r_0};\ w^0-\wz w^0=0\quad{\rm on}\ \partial B_{r_0}.$$
However, since $A\notin \Sigma (r_0)$, by Fredhlom alternative theorem,    $0$ is a unique solution to this equation.
Thus  $  w^0-\wz w^0=0$  in $B_{r_0}$, which is a contradiction.

 Finally, we show that $w^{0}\le 0$.  By \eqref{ef} and the convexity of $f$, we have
$$ \Delta u+Au-K=-f(u)+Au-K\le   f'_-(1)u-f'_-(1)+f(1)+Au-K=0.$$
By this and \eqref{inp1-1}, one has
$
-\bdz    w^0 \le  Aw^0
$, that is $-\Delta(-w^0)\ge A(-w^0)$ in weak sense in $B_{r_0}$.
Set $w^-=-\min\{0,-w^0\}$. Clearly, $w^{-}\in W^{1,2}_0(B_{r_0})$.
Multiplying $
-\bdz    w^0 \le  Aw^0$ with $w^{-}$ and integrating, by integration by parts and Sobolev inequality, one has
$$\int_{B_{r_0}}|Dw^-|^2\,dx\le A\int_{B_{r_0}}|w^-|^2\,dx\le
C_0(n)Ar^2_0\int_{B_{r_0}}|Dw^{-}|^2\,dx.$$
Since \eqref{rfn} gives
$C_0(n)Ar^2_0\le 1/8$, we have
  $w^-= 0$ in $B_{r_0}$ and hence $-w^0\ge 0$ as desired.
\end{proof}

 Lemma \ref{ex} allows us to prove Theorem 1.3 by  adapting the argument of \cite[Proposition 4.2]{cfrs}. We give the details as below.
\begin{proof}[Proof of Theorem \ref{app}]
Given any $x_0\in B_2$ let $r_0\in (0,1)$ be as in Lemma \ref{ex}.
Without loss of generality, we may assume that $x_0=0$ and $r_0=1$.
If  $\esssup_{B_1}f'_-(u)<+\fz$, then   by  the nondecreasing property  of $ f'_-$
(that is,  the convexity of $f$), we have
$$|f(u)|\le  |f(0)|+\max\{f'_-(0),\esssup_{B_1}f'_-(u)\}|u|\ a.\,e.\ {\rm in}\ B_1.$$
Thanks to this, by a standard elliptic theory (see for instance \cite{gt98}), we have $u\in C^2(B_1)$ as desired.
Below we always assume   $\esssup_{B_1}f'_-(u)=+\fz.$    Then one must have
$\esssup_{B_1}u=+\fz;$ otherwise, by
 the nondecreasing  property of $f'_-$,   one has
$$\esssup_{ B_1} f'_-(u )\le f'_-\left(\esssup_{ B_1}u \right)<+\fz.$$    We are going to obtain desired approximation by the following 6 steps.

{\bf Step 1.}
Given any $\ez\in(0,1)$,   set
\begin{align}\label{fe}
f_{\ez}(t):=\left\{
\begin{aligned}
&f(t)\quad&{\rm if}\ t<1/\ez ,\\
&f(1/\ez )+f'_-(1/\ez )(t-1/\ez )\quad&{\rm if}\ t\ge 1/\ez;
\end{aligned}
\right.
\end{align}
in other words,   $ f_\ez$ is obtained from $f$ via replacing $\{f(t)\}_{t\ge 1/\ez}$ by the left half of tangential line
  of $f$ at $ 1/ \ez$, that is, $\{f(1/\ez )+f'_-(1/\ez )(t-1/\ez )\}_{t\ge 1/\ez}$.

Obviously,    we have the following properties:
 \begin{align}\label{low00}
\mbox{ $f_\ez$
and $(f_\ez)'_-$ are nondecreasing, and $(f_\ez)'_-\le f'_-(1/\ez)$,}
\end{align}
and
 \begin{align}\label{low0}     \mbox{
$f\ge f_{\ez_1}\ge f_{\ez_2}$   whenever $0<\ez_1\le \ez_2\le 1$, }  \quad  f = \lim_{\ez\to0} f_\ez \quad  \mbox{and}\quad (f_\ez)'_-\le f'_-.\end{align}
 Consequently,   $f_{\ez}\in \mathcal{C}$ is Lipschitz function in $\rr$,  and hence
\begin{align}\label{low1}
\quad|f_\ez(t)|\le |f(0)|+ f'_-(1/\ez)|t| \quad  \forall t\in\rr.
\end{align}
Thanks to
 \eqref{low1},  $v\in W^{1,2}(B_1)$ implies
$ f_{\ez}(v)\in W^{1,2}(B_1) $. Moreover, since $f=f_\ez $ when $t<1/\ez$, by the convexity of $f_\ez$ we   have
\begin{align}\label{low}
 f_{\ez}(t)\ge (f_\ez)'_-(1)t-[(f_\ez)'_-(1)-f_\ez(1)]= At-K,\quad \forall t\in \rr,
\end{align}
where and below we write $A=f'_-(1)$ and $K=f'_-(1)-f(1)$.

{\bf Step 2.}  By Lemma \ref{ex},  there exists unique  $u^{(0)}\in C^2(B_1)$ satisfying
\begin{align}\label{in0}
-\bdz u^{(0)} =Au^{(0)}-K\quad&{\rm in}\ B_{1}, \quad
u^{(0)}=u\quad{\rm on}\ \partial B_{1};
\end{align}
moreover, $u^{(0)}\le u $ almost everywhere in $B_1$.
Given any $\ez\in(0,1]$, set $u^{(0)}_{\ez}:=u^{(0)}$ and
for any $j\ge 1$,  since $f_\ez$ is Lipschitz in $\rr$,
 there is a function  $u^{(j)}_\ez\in W^{1,2}(B_1)\cap C^2(B_1)$  satisfying
\begin{align}\label{inj}
-\bdz u^{(j)}_{\ez}=f_{\ez}(u^{(j-1)}_{\ez})\in W^{1,2}(B_1)\quad{\rm in}\ B_1;\quad
u^{(j)}_{\ez}=u\quad{\rm on}\ \partial B_1.
\end{align}
We claim that
\begin{itemize}
\item[(i)]     $u^{(0)}\le u^{(j)}_{\ez}\le
u^{(j+1)}_{\ez}\le u \quad\forall \ez\in(0,1],\ \forall j\ge 0;$
\item [(ii)]
 $u^{(j)}_{\ez_2}\le u^{(j)}_{\ez_1}\quad \mbox{ $\forall \ez_1,\ez_2\in(0,1]$ with  $\ez_1< \ez_2$,\  $\forall j\ge 0$; } $
\item[(iii)]  $u^{(j)}_{\ez}\in W^{1,2}(B_1)$ uniformly in
$\ez\in(0,1]$ and $j\ge 0$.
\end{itemize}

{\it Proof of Claim (i)}.
Thanks to \eqref{low0} and $u^{(0)}_\ez=u^{(0)} \le u$ one has
\begin{align*}
-\bdz( u-u^{(1)}_{\ez})=f(u)-f_{\ez}(u^{(0)}_{\ez})
=(f(u)-f_{\ez}(u))+(f_{\ez}(u)-f_{\ez}(u^{(0)}_{\ez}))\ge0
\end{align*}
in weak sense, that is, $u-u^{(1)}_{\ez}\in W^{1,2}_0(B_1)$ is superharmonic.
By the maximum principle, we get $u^{(1)}_{\ez}\le u$  in $B_1$.
On the other hand, with the aid of \eqref{low}  we obtain
\begin{align*}
-\bdz(u^{(1)}_{\ez}-u^{(0)}_{\ez})&=f_{\ez}(u^{(0)}_{\ez})-(Au^{(0)}_{\ez}-K)\ge0,
\end{align*}
and hence, by the maximum principle, we get $u^{(0)}_{\ez}\le u^{(1)}_{\ez}$.

For  any $j\ge2$, if $u^{(j-1)}_{\ez}\le u$,  similarly one has
\begin{align*}
-\bdz( u-u^{(j)}_{\ez})=f(u)-f_{\ez}(u^{(j-1)}_{\ez})
&=(f(u)-f_{\ez}(u))+(f_{\ez}(u)-f_{\ez}(u^{(j-1)}_{\ez}))\ge 0,
\end{align*}
and hence, by the maximum principle, we have $u^{(j)}_{\ez}\le u$.  On the other hand,  by \eqref{low0},
if $u^{(j-2)}\le u^{(j-1)}$, then
\begin{align*}
-\bdz(u^{j}_{\ez}-u^{(j-1)}_{\ez})=f_{\ez}(u^{(j-1)}_{\ez})-f_{\ez}(u^{(j-2)}_{\ez})\ge0,
\end{align*}
and hence, by the maximum principle again  one has $u^{(j-1)}_{\ez}\le  u^{(j)}_{\ez}$.

{\it Proof of Claim (ii)}.
Note that $u^{(0)}_{\ez_1}=u^{(0)}_{\ez_2}$.
For $j\ge 1$,  if $u^{(j-1)}_{\ez_2}\le u^{(j-1)}_{\ez_1}$,
  by \eqref{low0}
one has
\begin{align*}
-\bdz(u^{(j)}_{\ez_1}-u^{(j)}_{\ez_2})=f_{\ez_1}(u^{(j-1)}_{\ez_1})-
f_{\ez_2}(u^{(j-1)}_{\ez_2})\ge 0
\end{align*}
and hence  by the maximum principle, we obtain
 $u^{(j)}_{\ez_2}\le u^{(j)}_{\ez_1}$.

{\it Proof of Claim (iii)}.  Since  (i) implies $|u^{(j)}_{\ez}|\le  |u^{(0)}_{\ez}| +|u|$,
by $ u^{(0)}=  u^{(0)}_\ez  \in L^2(B)$ and  $u\in W^{1,2}(B)$, one knows that $|u^{(j)}_{\ez}|\in L^2(B)$ uniformly in $\ez>0$ and $j\ge0$.
Moreover, observe  that
\begin{align}\label{w12}
\int_{B_1}|Du^{(j)}_{\ez}|^2\,dx
&\le 2\int_{B_{1}}
Du^{(j)}_{\ez}\cdot (Du^{(j)}_{\ez}-Du)\,dx+ \int_{B_1}|Du|^2\,dx.
\end{align}
By  integration
by parts, $u^{(j)}_{\ez}-u\in W^{1,2}_0(B_{1})$ and
  \eqref{inj}, one has
\begin{align*}
&\int_{B_{1}}
Du^{(j)}_{\ez}\cdot (Du^{(j)}_{\ez}-Du)\,dx
=-\int_{B_{1}}f_{\ez}(u^{(j-1)}_{\ez})(u-u^{(j)}_{\ez})\,dx.
\end{align*}
Recalling that (i) gives $u-u^{(j)}_{\ez}\ge 0$  in $B_1$, and  \eqref{low} gives
$f_{\ez}(u^{(j-1)}_{\ez})\ge Au^{(j-1)}_{\ez}-K$, we obtain
\begin{align*}
&\int_{B_{1}}
Du^{(j)}_{\ez}\cdot (Du^{(j)}_{\ez}-Du)\,dx\le
-\int_{B_{1}}(Au^{(j-1)}_{\ez}-K)(u-u^{(j)}_{\ez})\,dx.
\end{align*}
Since $u^{(j)}_{\ez}\in L^2(B_1)$ uniformly in $j,\ez$,   we conclude
$D u^{(j)}_{\ez}\in L^{2}(B_1)$ uniformly in $j,\ez$.

{\bf Step 3.}    By  (iii)  and compactness of Sobolev spaces,
 we know that  for any $\ez\in(0,1]$,
 $u^{(j)}_{\ez}$  converges to some function $u_\ez$ in $L^2(B_1)$ and weakly in $W^{1,2}(B)$ as $j\to\fz$.
Thus
\begin{align*}
\int_{B_1}Du_{\ez}\cdot D\phi\,dx=\lim_{j\to\fz}\int_{B_1}Du^{(j)}_{\ez}\cdot D\phi\,dx
=\lim_{j\to\fz}\int_{B_1}f_{\ez}(u^{(j)}_{\ez})\phi\,dx.
\end{align*}
By the continuity of $f_\ez$, $\lim_{j\to\fz}f_{\ez}(u^{(j)}_{\ez})= f_{\ez}(u_{\ez})$.
Thanks to $(f_\ez)'_-\le f'_-(1/\ez)$, by  (i)  one has
\begin{align*}
|f_{\ez}(u^{(j)}_{\ez})|\le f'_-(1/\ez )|u^{(j)}_{\ez}|+|f(0)|
\le f'_-(1/\ez )(|u^{(0)}|+|u|)+|f(0)|\in L^2(B_1).
\end{align*}
Thus it follows by the Lebesgue's dominated convergence theorem that
\begin{align*}
\int_{B_1}Du_{\ez}\cdot D\phi\,dx =
\int_{B_1} f_{\ez}(u_{\ez})\phi\,dx.
\end{align*}
This implies that
  $u_{\ez} $ is a weak solution  to
  \begin{align}\label{inf}
-\bdz u_{\ez} =f_{\ez}(u_{\ez})\quad {\rm in}\ B_{1};
u_{\ez} =u\quad {\rm on}\ \partial B_{1}.
\end{align}
Thanks to $(f_\ez)'_-\le f'_-(1/\ez)$, by the standard elliptic theory, we know that  $u_\ez\in C^2(B_1)$ (see \cite{gt98}).

Since  $u_{\ez}\le u$, $(f_\ez)'_-\le f'_-$ and $(f_\ez)'_-$ is nondecreasing, we have
$$(f_\ez)'_-(u_{\ez})\le (f_{\ez})'_-(u)\le f'_-(u)$$
and hence
$$\int_{B_{r_0}}(f_\ez)'_-(u_{\ez})\xi^2\,dx
\le \int_{B_{1}}f'_-(u)\xi^2\,dx\le \int_{B_{1}}|D\xi|^2\,dx,$$
which implies that $u_\ez$ is a stable solution.

{\bf Step 4.} By (iii) and the definition of $u_\ez$, we know that $u_{\ez}\in W^{1,2}(B_1)$
uniformly in $\ez\in(0,1)$.  Thus  $u_{\ez}$ converges to some function $  u^{\star}$ in $L^2(B_1)$ and
 weakly in $L^2(B_1)$ as $\ez\to 0$.
Since (ii) implies that  $u^{(0)}\le u_{\ez_2}\le u_{\ez_1}\le u$ whenever $0<\ez_1<\ez_2<1$,
 we know that $u^\star\le u$ a.\,e. in $B_1$. Note that $u^{\star}-u\in W^{1,2}_0(B_1)$.

Next we show that $u^{\star} $ is a weak solution to
 \begin{align}\label{ins}
 -\bdz u^{\star} =f(u^{\star})\quad {\rm in}\ B_{1},\quad
u^{\star} =u\quad &{\rm on}\ \partial B_{1}.
\end{align}
Considering \eqref{inf} and  $Du_{\ez}\to Du^{\star}$ weakly in $L^2(B_1)$ as $\ez\to 0$, we only need to prove that
\begin{align}\label{xx2}\lim_{\ez\to 0}\int_{B_1}f_{\ez}(u_{\ez})\,dx=
\int_{B_1}f(u^{\star})\,dx.\end{align}
To see this, write
\begin{align}\label{xx1}
 &\int_{B_1}|f_{\ez}(u_{\ez})-f(u^{\star})|\,dx\nonumber\\
&=\int_{B_1\cap \{u<1/\ez \}}|f(u_{\ez})-f(u^{\star})|\,dx
+\int_{B_1\cap \{u\ge 1/\ez \}}|f_{\ez}(u_{\ez})-f(u^{\star})|\,dx.
\end{align}
By \eqref{low} and (i), we get
\begin{align}\label{fb2}
|f_{\ez}(u_{\ez})|+|f(u^{\star})|+|f(u_{\ez})|&\le 6[|f(u)|+A|u^{(0)}|+A|u|+K]\in L^1(B_{1}),
\end{align}
Thanks to this and $f(u_{\ez})\to f(u^{\star})$  a.\,e.,  the first term in the right-hand side of \eqref{xx1} tends to zero.
Moreover due to \eqref{fb2} again and  the absolute continuity of integrals, the second term in right-hand side of \eqref{xx1} also tends to zero.
Thus \eqref{xx2} holds as desired.

{\bf Step 5.}  To show   $u=u^{\star}$ a.\,e. in $B_1$, we argue by contradiction. Assume that   $u\not= u^{\star}$ a.\,e. in $B_1$.
Since  $u^{\star}\le u$ a.\,e. in $B_1$ and $f$ is nondecreasing, we have
$$-\bdz (u-u^{\star})=f(u)-f(u^{\star})\ge 0\quad\mbox{ in weak sense},$$
that is, $ u-u^{\star}\in W^{1,2}_0(B_1)$ is superharmonic  in $B_1$.
Since $u\not\equiv u^{\star}$, by the strong maximum principle we have $u-u^\star>0$ a.\,e. in $B_1$.

Since the  convexity of $f$ gives $$f(u^{\star})-f(u)\ge f'_-(u)(u^{\star}-u),$$
by $u^{\star}\le u$ a.\,e. in $B_1$ one has
\begin{align*}
 (f(u)-f(u^{\star}))(u-u^{\star})  \le  f'_-(u)(u-u^{\star})^2 \quad\mbox{a.\,e. in $B_1$ }.
\end{align*}

  On the other hand, applying the stability inequality to  $u-u^{\star}\in W^{1,2}_0(B_1)$ (up to some  approximation via functions in $C^{0,1}_c(B_1)$),
recalling
$
-\bdz (u-u^{\star})=f(u)-f(u^{\star}) $ in $B_1$,  we have
\begin{align*}
 \int_{B_1}f'_-(u)(u-u^{\star})^2\,dx\le\int_{B_1}|D(u-u^{\star})|^2\,dx&=
\int_{B_1}(f(u)-f(u^{\star}))(u-u^{\star})\,dx.
\end{align*}
We then conclude that
$$(f(u)-f(u^{\star}))(u-u^{\star})=f'_-(u)(u-u^{\star})^2\quad{\rm a.e.\ in}\ B_1,$$
and hence
\begin{align*}
f(u)-f(u^{\star})=  f'_-(u) (u-u^{\star}) \quad\mbox{a.\,e. in $B_1$ }.
\end{align*}
By the convexity of $f$, we know that
$f$ is linear in the interval $ [u^\star(x),u(x)]$ for a.\,e. $x\in B_1$.
Via the same argument as in \cite[Proposition 4.2]{cfrs} and \cite[Theorem 16]{df}, one concludes that $f$ is linear on $(\essinf_{B_1}u^{\star},\esssup_{B_1}u)=(\essinf_{B_1}u^{\star},+\fz)$.
That is, $f'_-\le C$ for some constant $C<+\fz$.
This contradicts with
  $\esssup_{B_1}f'_-(u)=+\fz$.
Therefore  $u=u^{\star}$ a.e in $B_1$.

{\bf Step 6.} We claim that $u_{\ez}\to u$ in $W^{1,2}(B_1)$ as $\ez\to0$. Indeed,
thanks to Step 5,
by (iii) we know that $u_{\ez}\to u$ in $L^2(B_1)$ and
$Du_{\ez}\to Du$ weakly in $L^2(B_1)$. It remains to show that
\begin{align}\label{cw12}
\limsup_{\ez\to 0}\int_{B_1}|Du_{\ez}|^2\le \int_{B_1}|Du|^2\,dx.
\end{align}
Testing \eqref{inf} with $u-u_{\ez}\in W^{1,2}_0(B_{1})$ we obtain
\begin{align*}
\int_{B_{1}}Du_{\ez}\cdot (Du-Du_{\ez})\,dx=
\int_{B_{1}}f_{\ez}(u_{\ez})(u-u_{\ez})\,dx.
\end{align*}
It follows by $u^{(0)}\le u_{\ez}\le u$, \eqref{low00} and \eqref{low} that
\begin{align*}
\int_{B_{1}}|Du_{\ez}|^2\,dx&=
\int_{B_{1}}Du_{\ez}\cdot Du\,dx-\int_{B_{1}}f_{\ez}(u_{\ez})(u-u_{\ez})\,dx\\
&\le \int_{B_{1}}Du_{\ez}\cdot Du\,dx-\int_{B_{1}}(Au^{(0)}-K)(u-u_{\ez})\,dx\\
&\le \int_{B_{1}}Du_{\ez}\cdot Du\,dx
+\left(\int_{B_{1}}(Au^{(0)}-K)^2\,dx\right)^{\frac12}
\left(\int_{B_{1}}(u_{\ez}-u)^2\,dx\right)^{\frac12}.
\end{align*}
From this, thanks to $u^{(0)},u\in W^{1,2}(B_1)$, using $u_{\ez}\to u$ in $L^2(B_1)$ and
$Du_{\ez}\to Du$ weakly in $L^2(B_1)$ one gets \eqref{cw12}.
This implies that $u_{\ez}\to u$ in $W^{1,2}(B_1)$ as $\ez\to0$.

\end{proof}

{\bf Acknowledgement.}  The authors would like to thank Professor Xavier Cabr\'e for several valuable comments and suggestions in the previous version of this paper.

\noindent  Fa Peng

\noindent
Academy of Mathematics and Systems Science, the Chinese Academy of Sciences, Beijing 100190, P. R. China

\noindent{\it E-mail }:  \texttt{fapeng@amss.ac.cn}

\bigskip

\noindent Yi Ru-Ya Zhang

\noindent
Academy of Mathematics and Systems Science, the Chinese Academy of Sciences, Beijing 100190, P. R. China

\noindent{\it E-mail }:  \texttt{yzhang@amss.ac.cn}

\bigskip

\noindent  Yuan Zhou

\noindent
School of Mathematical Science, Beijing Normal University, Haidian District Xinjiekou Waidajie No.19, Beijing 10875, P. R. China

\noindent{\it E-mail }:  \texttt{yuan.zhou@bnu.edu.cn}

\end{document}